\documentclass{amsart}
\usepackage{amsmath, amssymb, amsthm,graphicx}
\usepackage{setspace}
\usepackage[utf8]{inputenc}
\usepackage[english]{babel}
\usepackage[normalem]{ulem}
\usepackage{soul}
\usepackage{mathrsfs}
\usepackage{verbatim}
\usepackage{tikz}
\usetikzlibrary{shapes.geometric}

\usetikzlibrary{fit,calc,positioning,decorations.pathreplacing,matrix}
\usetikzlibrary{arrows,decorations.markings}

\numberwithin{equation}{section}
\theoremstyle{plain}
\newtheorem{thm}{Theorem}[section]
\newtheorem{prp}[thm]{Proposition}

\newtheorem{lm}[thm]{Lemma}
\newtheorem{rmq}[thm]{Remark}

\theoremstyle{definition}
\newtheorem{df}[thm]{Definition}

\newcommand\WS{(W,S)}
\newcommand\FF{\mathcal{F}\WS}
\newcommand\FFF{\mathcal{F}}
\newcommand\OWS[1]{\Omega(#1)}
\newcommand\OWSs[1]{\Omega_0(#1)}
\newcommand\OWSp[1]{\Omega_1(#1)}
\newcommand\CSL{\Lambda\WS}
\newcommand\CSLs{\Lambda}
\newcommand\rond[1]{\mathring#1}
\title{Minimal generating set  of Cactus groups}
\author{Eddy Godelle}
\address{Godelle Eddy, Universit\'e de Caen, LMNO UMR 6139  14000 Caen, France}
\email{eddy.godelle@unicaen.fr}
\begin{document}

\begin{abstract}
\noindent
We  prove that the lower central series of the cactus group associated with  a non commutative Coxeter group never stabilizes. We also compute a minimal presentation in terms of generators for the cactus group associated with a finite Coxeter groups, except in type $E$.
\end{abstract}
. 
\maketitle

The first appearance of the cactus group $J_n$ is implicit in \cite{Dri1990}. It was explicitly  and independently introduced  in \cite{Dev1999} and \cite{HenKam2006}  where it is related to some configurations spaces, operads and coboundary categories. More generally, a group  $C(W,S)$  can be associated  with every Coxeter system $(W,S)$ \cite{DavJanSco2003}. It is  still called a cactus group. The cactus group $J_n$ is the cactus group associated with the symmetric  group $\mathfrak{S}_n$ equiped with its classical Coxeter structure. Recently cactus groups $C(W,S)$ have attracted the attention of specialists in representation  theory \cite{Bon2016, Los2019, ChmGliPyl2020,RouWhi2024}. In particular, they are expected to be related to  the Calogero-Moser spaces and  to the Kazhdan-Lusztig cells~\cite{RouWhi2024}.  Recall~\cite{Bou} that a Coxeter matrix on a finite set $S$ is a symmetric matrix $M_S = (m_{s,t})_{s,t\in S}$  where diagonal entries are equal to~$1$ and nondiagonal entries lie in $\mathbb{N}_{\geq 2}\cup \{\infty\}$. Its  associated Coxeter group $W$ is defined by the group  presentation  \[
W=\langle S\mid s^2=1; Prod(t,s,m_{s,t}) = Prod(s,t,m_{s,t})\text{ for }s,t\in S\,,\ s\neq t\,,\ m_{s,t}\neq\infty\rangle\,\] where  $Prod(s,t,m)$ denotes the word $sts\cdots$ with $m$ letters.  The pair  $\WS$ is called the Coxeter system associated with the Coxeter matrix~$M_S$.  For any subset $X$ of $S$, by $W_X$ we denote the subgroup of $W$ generated by $X$. Such a subgroup is called a standard parabolic subgroup of $W$. This is well-known that the pair $(W_X,X)$ is the Coxeter system associated with the matrix $(m_{s,t})_{s,t\in X}$. The Coxeter group $W$ is said to be irreductible if it can not be written as a not trivial direct product of two of its standard parabolic subgroups.  This is equivalent to say there is no proper partition $X\cup Y$ of $S$ such that  $m_{x,y} = 2$ for any  $x\in X$ and $y\in Y$.  When $W$ is finite and irreducible,  there exists a unique element $\omega_S \neq 1$ in $W$ that permuts $S$ by conjugation.  Moreover $\omega_X$  has order $2$.  Irreducible Coxeter systems with $W$ finite are classified (see Section~\ref{annexepresent}). By $\FF$, or simply $\FFF$,  let us denote  the set of all non-empty subsets $X$ of $S$ such that $W_X$ is a finite, irreducible parabolic subgroup of $W$.   For $X\in\FFF$,  set  $$\OWS{X} = \{Y\in \FFF  \mid  Y\neq X \textrm{ and } \omega_X Y \omega_X \subseteq S \}.$$  
We have a bijection $\omega_X:\OWS X \to \OWS X$  of order~$1$ or~$2$  defined by $\omega_X(Y) = \omega_XY\omega_X$.  In this case, we have $\omega_X \omega_Y \omega_X = \omega_{\omega_X(Y)}$ (see Lemma~\ref{lemme1}(i)). Using the same notation to denote both the element $\omega_X$ in $W$ and the associated bijection $\omega_X \omega_Y \omega_X = \omega_{\omega_X(Y)}$ is an abuse of notation, but this will not cause any confusion. There is a partition $\OWS{X} = \OWSs{X} \cup \OWSp{X}$  where  $$\begin{array}{l}\OWSs{X} =  \{Y\in \FFF  \mid Y\subsetneq X\}\\\OWSp{X} =  \{Y\in \FFF  \mid Y\cup X\textrm{ not irreducible}\}\end{array}. $$ Clearly,  $\omega_X$ stabilizes $\OWSs{X}$ and fixes  $\OWSp{X}$:  when $Y\in\OWSs{X}$, then $\omega_X(Y)$ lies in  $\OWSs{X}$;   when $Y\in\OWSp{X}$, then $\omega_X(Y) = Y$.  Moreover $Y\in\OWSp{X} \iff X\in\OWSp{Y}$.

\begin{df}\label{defcactus}
The Cactus  group $C\WS$ associated with the Coxeter system  $\WS$ is defined by  the following group presentation: 
\begin{equation} \label{pres1cactus}
\left\langle c_X,\  X \in \FFF\ \left | \begin{array}{cccl}
(R_1)&c_X^2=1&;&X\in \FFF\\
(R_2)&c_X c_Y =  c_{\omega_X(Y)} c_X&;&Y\in \OWSs{X}\\ 
(R_3)& c_Y c_X = c_X c_Y&;&Y\in \OWSp{X}
  \end{array}\right.\right\rangle .
\end{equation}  
\end{df}
   
For the remaining of the article, we set $C_\FFF = \{c_X \mid X\in \FFF\}$. More generally,  for $U \subset \FFF$ we set $$C_U = \{c_X \mid X\in U\}.$$ We remark there is a relation $c_Xc_Y = \cdots$ if and only if there is a relation $c_Yc_X = \cdots$, and this happens  precisely when $Y\in \OWS X$ or $X\in \OWS Y$.   It immediatly follows from the presesentation that the map $c_X\mapsto \omega_X$ extends to an onto morphism from $C\WS$ to~$W$.  As a consequence, the $c_X$ are distincts in $C\WS$. 
 
In  \cite{CheNan2025},  the authors adress combinatorial questions about the classical cactus groups $J_n$, such as finding a minimal presentation in terms of generators  and  the study of its possible finite quotients in connection with the lower central series of the cactus group.  The objective of the present article is  both to extend their results to other cactus groups and to provide a short proof in the case $J_n$. The following results extend \cite[Theorem B]{CheNan2025} and partially \cite[Theorem C]{CheNan2025}.

\begin{prp} \label{ThmA3} Let $\WS$ be a Coxeter system. Consider the equivalence relation $\equiv$ on $\FF$ defined as the transitive closure of the binary relation $\equiv_0$ defined by $Y\equiv_0 Z$ is there exists $X$ so that $Z = \omega_X\,Y\, \omega_X$ in $W$.
Denote by $m$ the number of equivalent classes on $\FFF$. Then,
\begin{enumerate}
\item The abelianisation group of $C\WS$ is isomorphic to $\mathbb{Z}^m_2$.
\item  Any generating set of $C\WS$ possesses at least $m$ elements. 
\item If $\Lambda$ is a transversal for $\equiv$, then $C\WS$ possesses a finite presentation with generating set~$\Lambda$. 
\end{enumerate} 
\end{prp}
  
\begin{thm}\label{ThmBC}
 Let $\WS$ be a  Coxeter system. For $U\subseteq \FFF $, by $C(W,U)$  we denote the subgroup of $C\WS$ generated by the set $C_U$.
\begin{enumerate}
\item  The group $C\WS$ is  Abelian if and ony if $W$ is Abelian.
\item When $W$ is not Abelian then there exist $X,Y\in \FFF$ and a subgroup $G$ of $C\WS$ such that 
\begin{enumerate}
\item  $C(W,\{X,Y\})$ is isomorphic to $\mathbb{Z}_2 * \mathbb{Z}_2$.
\item  $C\WS = G\rtimes C(W,\{X,Y\})$. 
\item The lower central series of $C\WS$ does not stabilize. 
\end{enumerate} 
\end{enumerate} 
\end{thm}

At this point, a natural question is whether some transversals provide better presentations. The following definition and theorem aim to answer this question, generalising \cite[Theorem A]{CheNan2025}. 

\begin{df} \label{defiCSL} Let $\WS$ be a Coxeter system.  Consider  a subset $\CSLs$ of $\FFF$.
\begin{enumerate} 
 \item A map $\Psi: \FFF \to \CSLs\times \CSLs, X\mapsto (\overline{X},\rond{X})$ is said to be a \emph{section map} for $\WS$ when
\begin{enumerate}
\item  For all $X \in \CSLs$ we have $\rond{X} = \overline{X} = X$; 
\item  For all $X \in \FFF$ we have $\rond{X}\subseteq \overline{X}$ and $\omega_{\overline{X}}(\rond{X}) = X$. 
\item For any $Y,Z$ in $\FFF$  with $Y\cup Z$ not irreducible, there exists $X\in \CSLs$ so that $Y,Z \in \OWS X$ and $\{\omega_X(Y), \omega_X(Z)\}\cap \CSLs \neq\emptyset$.   
\end{enumerate}  
In this case, we say that the pair $(\CSLs,\Psi)$ is a section for $\WS$.
\item A section $(\CSLs,\Psi)$ is called a transversal section  when~$\CSLs$ is a transversal for the relation~$\equiv$.
\item   A subset  $\CSLs$  is said to be a \emph{cross section} for $\WS$ when it possesses a  section map $\Psi$ so that for all $X$ in $\FFF$ the pair $\Psi(X) = (Y,Z)$ is the unique pair of elements of $\CSLs$ so that $\omega_Y(Z) = X$. 
\end{enumerate} 
\end{df}
Clearly, when $\CSLs$ is a cross section, then $\Psi$ is uniquely defined and $(\CSLs,\Psi)$ is a transversal section.
Note that $\WS$ is always equiped  a section map:  the map $X\mapsto (X,X)$ is a section map.  However, more can be said when $W$ is finite and irreducible:  
\begin{prp} \label{propA} Let $\WS$ be a Coxeter system with $W$ finite and irreducible. Then,\begin{enumerate}
\item If $W$ is of type $A$, $B$, $D_{2n}$, $F_4$, $I_n$, $H_3$ or $H_5$ then $\WS$ possesses a cross section.
\item If  $W$ is of type $D_{2n+1}$ then $\WS$ possesses a transversal section.
\item if $W$ is  of type $E_6$, $E_7$ or $E_8$, then $\WS$  does not possess a transversal section.  
\end{enumerate} 
\end{prp}

\begin{thm}\label{ThmA2} Let $\WS$ be a Coxeter system and $(\CSLs,\Psi)$ be a cross section.  Then $C\WS$ possesses the following group presentation: 
\begin{equation} \label{pres3cactus}
\left\langle C_\CSLs \left | \begin{array}{cccl}
(R_{1.a})&c_X^2=1&;&X\in \CSLs;\\
(R_{2.b})& c_Xc_Yc_Zc_Y =c_{Y'}c_{Z'}c_{Y'} c_X&;&\left\{\begin{array}{l}X\in \CSLs ; (Y,Z), (Y',Z') \in \Psi(\Omega_0(X)\!\setminus\!\CSLs)\\\textrm{and } \omega_{Y'}(Z') = \omega_{X}(\omega_{Y}(Z))\end{array}\right.\\
 (R_{3.b})&(c_X c_Yc_Xc_Z)^2 = 1&;&  Z\in \CSLs \textrm{ and }  (X,Y)\in \Psi\big(\OWSp{Z}\setminus \CSLs\big) \end{array}\right. \right\rangle
\end{equation}
Moreover,  the above presentation is minimal in terms of generators.
\end{thm}
 In Section~\ref{Section2} we prove  Proposition~\ref{ThmA3} and  Theorem~\ref{ThmBC}.
In Section~\ref{Section1} we prove  Theorem~\ref{ThmA2}. Indeed we provide a presentation for any section~$(\CSLs,\Phi)$; see Theorem~\ref{ThmA}.
In Section~\ref{annexepresent} we prove Proposition~\ref{propA}. In  particular for each finite irreducible Coxeter system that is not of type $E$  we provide a cross section or a transversal section, according to the proposition.

\section{Abelinanisation} \label{Section2}

In this section  we prove  Proposition~\ref{ThmA3} and  Theorem~\ref{ThmBC}. We  first recall some notation:  the binary relation $\equiv_0$ is  defined by $Y\equiv_0 Z$ is there exists $X$ so that $Z = \omega_X\,Y\, \omega_X$ in $W$. By $\equiv$ we denote the equivalence relation on $\FF$ defined as the transitive closure of the binary relation~$\equiv_0$.  We start with the proof of Proposition~\ref{ThmA3}. 

\begin{proof}[Proof of Proposition~\ref{ThmA3}]
Let $\WS$ be a Coxeter system. Consider  $\equiv$ and $\equiv_0$ as defined in the proposition. The  defining relations of Presentation~(\ref{pres1cactus}) fall into two categories:  torsion relations (those of type  ($R_1$)) and quadratic relations $c_Xc_Y = c_Zc_X$ (those of types~($R_2$) or~($R_3$)). In the latter case we have $Y\equiv_0Z$. Conversely, when $Y\equiv_0Z$ with $Y,Z$ distinct, there exits $X$ so that both~$Y,Z$ belong to $\Omega(X)$ and $\omega_X(Y) = Z$ with a defining relation  $c_Xc_Y = c_Zc_X$ of type $(R_2)$ or $(R_3)$.  So the abelianisation of $C\WS$ leads to identify any two $c_Y, c_Z$ so that $Y\,\equiv\,Z$.  The remaining no-trivial defining relations are those of type $(R_1)$ and commuting relations. So Point~(i) of  Proposition~\ref{ThmA3} holds. Point~(ii) follows immediately. We turn to the proof of Point~(iii).  Let $\Lambda$ be a transversal for the equivalence classes for $\equiv$. By Tiezte's result on Tiezte transformations~\cite{MagnusCS},  it is enough to prove that for any element $X$ of $\FFF$ which is not in $\CSLs$  there exists a relation $c_X= w$, where $w$ is a word on $\{c_X\mid X\in \Lambda\}$, that can been obtained as a consequence of the defining relations of Presentation~(\ref{pres1cactus}). By $\CSLs_0$ denote  the set of elements $X$ in $\FFF$ that either are in $\CSLs$ or satisfy the latter property.  Let $Z$ be in $\FFF$ and let us prove that $Z$ belongs to $\CSLs_0$. Since the equivalence relation $\equiv$ is the transitive closure of the binary relation $\equiv_0$ and $\CSLs\subseteq \CSLs_0$, we are reduce to prove that if $Y \equiv_0 Z$ and $Y\in \CSLs_0$, then $Z$ lies also in $\CSLs_0$. Let us prove it  by induction on the cardinality $m$ of $S\setminus Z$. For $m = 0$, we have $Z = S$. Since $Z$ is alone in its $\equiv$-class, it has to belong to $\CSLs$ and there is nothing to prove. Assume $m\geq 1$. If $Y= Z$ then, there is nothing to prove. So assume this is not the case. Then there exists $X$ in $\FFF$ so that $Z = \omega_X(Y)$ with $Y,Z\in \Omega_0(X)$ and $c_Xc_Y = c_Zc_X$ is a relation of type ($R_2$). Using the relation $c_X^2 = 1$ of type  ($R_1$) we get $c_Z = c_Xc_Yc_X$. But $Y,Z$ are distinct and included in $X$, so the cardinality  of $S\setminus X$ is smaller than $m$. By the induction hypothesis, $X$ belongs to $\CSLs_0$. Using the obtained relation  $c_Z = c_Xc_Yc_X$, we conclude that $Z$ is in $\CSLs_0$. Hense, $\CSLs_0 = \FFF$ and Point~(iii) is proved.    \end{proof}

Before proving Theorem~\ref{ThmBC}, let us recall some classical notions \cite{Mik2009}. The commutator $[g,h] $ of two elements $g,h$ in  a group~$G$ is  $[g,h] = ghg^{-1}h^{-1}$.  The  terms of the lower central series of  a group~$G$ are defined inductively by setting $\Gamma_1(G) = G$ and $\Gamma_{n+1}(G) = [G,\Gamma_n(G)]$,   that is the subgroup of $G$ generated by the set $\{[g,h] \mid g\in G ; h\in\Gamma_n(G)\}$ of commutators of $G$.  It is immediate by induction that $\Gamma_{n+1}(G)$ is  both a normal subgroup of $G$ and a subgroup of $\Gamma_{n}(G)$,  and that the quotient groups $\Gamma_{n}(G)/\Gamma_{n+1}(G)$ are abelian.  One says that the lower central series of group~$G$ stabilizes if there exists  some $n$ such that $\Gamma_{n+1}(G) = \Gamma_{n}(G)$,  that is  $\Gamma_{n}(G)/\Gamma_{n+1}(G) = 1$.  Clearly in this case, one has $\Gamma_{m}(G) = \Gamma_{n}(G)$ for any $m\geq n$. One says that $G$ is (simply) nilpotent if $\Gamma_{n}(G) = \{1\}$ for some $n$. One says that 
$G$ is residually nilpotent if $\cap_{n\geq 0} \Gamma_{n}(G) = \{1\}$.  For instance  consider the group~$G = \mathbb{Z}_2 * \mathbb{Z}_2 =  \langle \begin{array}{l} u,v \mid   u^2 =  v^2  = 1\end{array}\rangle$. One can compute by hand that  for any $n\geq 2$ one has $\Gamma_{n}(G)$ is generated by $(uv)^{2^{n-1}}$. So, the lower central series of $G$  does not stabilizes,  and $G$ is, therefore, not (simply) nilpotent. But it is residually nilpotent.

\begin{proof}[Proof of Theorem~\ref{ThmBC}]
Recall from the introduction that we have a morphism $C\WS\to W,\  c_X\mapsto \omega_X$ that is onto. Then, if $C\WS$ is abelian, so is $W$. Conversely, if $W$ is abelian then for any $X$ in $\FFF$ we have $\FFF = \Omega(X) \cup \{X\}$, which means that $\omega_X(Y) = Y$ for any $Y\in \FFF$. So for any two distinct $X,Y\in\FFF$, the relation $c_Xc_Y = c_Yc_X$ is a defining relation in Presentation~(\ref{pres1cactus}). Thus $C\WS$ is abelian.
Assume $W$ is not Abelian and let us prove~(ii).  We can assume $W$ irreducible without restriction: if $W = W_{X_1}\times \cdots \times W_{X_k}$, then $C\WS = C(W_{X_1},X_1) \times \cdots \times C(W_{X_k},X_k)$ and one of the terms of the decomposition is not abelian. Points (a) and (b) of (ii) immediately follow. Point (c) follows from the following property: if $G = G_1\times G_2$, then it is immediate that $\Gamma_n(G) = \Gamma_n(G_1)\times \Gamma_n(G_2)$ for any $n$.  Write $\mathbb{Z}_2 * \mathbb{Z}_2 =  \langle \begin{array}{l} u,v \mid   u^2 =  v^2  = 1\end{array}\rangle$. Let us prove  Points~(a) and~(b), first.  As a warm up we start with the  particular case where $W$ is  finite dihedral, that is of type $I_n$. Set $S = \{s,t\}$. We have  $C_\FFF = \{c_s, c_t, c_{\{s,t\}}\}$. If $m_{s,t}$ is even,  the presentation of $C\WS$  is  $\left\langle C_\FFF\left|  \begin{array}{l}  c^2_s =  c^2_t =  c^2_{\{s,t\}} = 1\\  c_s\, c_{\{s,t\}} =  c_{\{s,t\}} \,c_s \\ c_t\, c_{\{s,t\}} =  c_{\{s,t\}} \,c_t\\ \end{array}\right.\right\rangle$. If $m_{s,t}$ is odd, the presentation is  $\left\langle C_\FFF \left | \begin{array}{l}  c^2_s =  c^2_t =  c^2_{\{s,t\}} = 1\\  c_s\, c_{\{s,t\}} =  c_{\{s,t\}} \,c_t \\ c_t\, c_{\{s,t\}} =  c_{\{s,t\}} \,c_s\end{array}\right.\right\rangle$. In the first case,  one has a morphism that sends $c_s$ on $u$, $c_t$ on $v$ and $c_{\{s,t\}}$ on $1$. Clearly we get a section by setting  $u\mapsto c_s$ and $v\mapsto c_t$. Therefore, the subgroup of $C\WS$ generated by $c_s$ and $c_t$, that is~$C(W,\big\{\{s\},\{t\}\big\})$,  is isomorphic to $\mathbb{Z}_2 * \mathbb{Z}_2$ and we have the expected semi-direct product  $C\WS = G\rtimes C(W,\big\{\{s\},\{t\}\big\}))$, where $G$ is the kernel of the above morphism onto $\mathbb{Z}_2 * \mathbb{Z}_2$.  Actually, $G$ is generated by $c_S$ and $C\WS = C(W,\{S\}) \times C(W,\big\{\{s\},\{t\}\big\})$.  In the second case,  one can send $c_s$ on $u$, $c_{\{s,t\}}$ on $v$ and $c_t$ on $vuv$. Similarly to the previous case, the subgroup of $C\WS$ generated by $c_s$ and $c_{\{s,t\}}$, that  is~$C(W,\big\{\{s\},\{s,t\}\big\})$, is isomorphic to $\mathbb{Z}_2 * \mathbb{Z}_2$ and  we have a semi-direct product $G\rtimes  C(W,\big\{\{s\},\{s,t\}\big\})$ where $G$ is the kernel of the morphism. Actually, in this case, this kernel is trivial and $C\WS = C(W,\big\{\{s\},\{s,t\}\big\})$.  We come back to the general (irreducible non Abelian) case. Assume that either  we have $W$ of spherical type with $\omega_S$ central in $W$ or we have $W$ that is not of spherical type. Consider two distinct  elements $X,Y$ of  $\FFF\setminus \{S\}$ and maximal for the inclusion. Note that such a pair exists: take $X$ in $\FFF$ distinct from $S$ and maximal. Consider $y\in S\setminus X$.  Then $\{y\}$ belongs to $\FF$ and there exists a maximal element of $\FFF\setminus \{S\}$ that contains $y$.  Then,  we can conclude as for the even dihedral case: considerer the morphism  from $C\WS$ onto $\mathbb{Z}_2 * \mathbb{Z}_2$ that sends every element $c_Z$ of the generating set on $1$,  except $c_X,c_Y$ that are sent on $u$ and $v$ respectively.  Assume  finally $W$ is of spherical type  and $\omega_S$ is not central in $W$.  Considering the classification of irreducible finite Coxeter groups  (see Section~\ref{annexepresent}), in addition to the odd dihedral case, there is only three possible cases: $W$ is of type $A$, $D_{2n+1}$ or $E_6$. In each case, there exist two distinct  maximal proper irreducible parabolic subgroups $W_X$ and $W_Y$  so that $X\omega_S  = \omega_S Y$  and we can conclude as for the the odd dihedral case:  considerer the morphism  from $C\WS$ onto $\mathbb{Z}_2 * \mathbb{Z}_2$ that sends every element $c_Z$ of the generating set on $1$,  except $c_X,c_Y$ and $c_S$ that are sent on $u$ and $vuv$ and $v$, respectively. So in any cases, Points~(a) and~(b) hold.  Let us now prove Point~(c). If $\varphi: G \to H$ is a morphism of groups, it is obvious by induction that for any $n$ one has $\varphi(\Gamma_n(G)) \subseteq \Gamma_n(H)$ and we have induced morphisms $\varphi$ from $\Gamma_{n}(G)$ and  $\Gamma_{n}(G)/\Gamma_{n+1}(G)$ to $\Gamma_{n}(H)$  and  $\Gamma_{n}(H)/\Gamma_{n+1}(H)$, respectively. When moreover  the morphism $\varphi: G \to H$ is onto, then so are the induced morphisms. Therefore if $G$ is (simply) nilpotent, so is $H$ ;  if  the lower central series  of $G$ stabilizes, then the one of $H$ stabilizes too.
 But as seen above the proof of Theorem~\ref{ThmBC}, the lower central series of  $\mathbb{Z}_2 * \mathbb{Z}_2$ does not stabilize. Thus, we are done. \end{proof}

As far as we know the question of whether $C\WS$ is residually nilpotent  is open, even in the case of the classical cactus group $J_n$. We remark that the answer is positive for the (dihedral) Coxeter group of type $I_n$, since, as seen along the above proof,  in this case $C\WS$ is either $(\mathbb{Z}_2 * \mathbb{Z}_2)\times \mathbb{Z}_2$ or $\mathbb{Z}_2 * \mathbb{Z}_2$, depending whether $n$ is even or odd. 

\section{Cross section} \label{Section1}
For all the section we fix  a Coxeter system $\WS$ and a section $(\CSLs,\Psi)$ for $W$.
 The proof of Theorem~\ref{ThmA2} (and of Theorem~\ref{ThmA} below) is an  application of Tiezte's result on Tiezte transformations~\cite{MagnusCS}.  Indeed, under the hypothesese of the theorem,  for all $Z$ in $\FFF$, there exist $X,Y \in \CSL$ with $Y\subseteq X$ so that $\omega_X(Y) = Z$, namely $X = \overline{Z}$ and $Y = \rond{Z}$.  So  we have $c_{\overline{Z}}\,c_{\rond{Z}}= c_Z\,c_{\overline{Z}}$ and, equivalently, $c_{\overline{Z}}\,c_{\rond{Z}}\,c^{-1}_{\overline{Z}} = c_Z$. This means that  the set  $\CSLs$ generates $C\WS$. Using Tiezte transformations, from the defining presentation~(\ref{pres1cactus}) of $C\WS$,  we are going to deduce a finite presentation of $C\WS$ with  $\CSLs$ for generating set.

 Note that the relations  $c_{\overline{Z}}\,c_{\rond{Z}}= c_Z\,c_{\overline{Z}}$ belong to relations of type (R2), except if $\rond{Z} =\overline{Z}$.  This latter case  only happen for $Z$ in $\CSLs$.  Among the relations  $c_X c_Y =  c_{\omega_X(Y)} c_X$ of type ($R_2$)  we call of type~($\widehat{R}_{2}$) those such that $X$ belongs to~$\CSLs$ and precisely only one among  $Y$ and $\omega_X(Y)$ belongs to~$\CSLs$.  We call of type~($\widehat{R}_{2.c}$) any relation $c_X c_Yc_X =  c_{\omega_X(Y)}$ with $Y\in \Omega_0(X)$  and such that both $X$ and $Y$ belong to~$\CSLs$ but $\omega_X(Y)$ does not. 

\begin{lm}\label{lemme1} 
\begin{enumerate}
\item For $Y\subseteq X$ in $\FFF$, one has  $\omega_X(\omega_Y) = \omega_{\omega_X(Y)}$;
\item For any $Z\subseteq Y\subseteq X$ in $\FFF$,  one has  $\omega_X(\omega_Y(Z)) = \omega_{\omega_X(Y)}(\omega_X(Z))$
\end{enumerate}
\end{lm}
\begin{proof} 

Recall also that for any $Z\in \FFF$, the element $\omega_Z$ is the unique non trivial element of $W_Z$ that permutes $Z$  by conjugacy in $W$. Since $\omega_X$ permutes $X$ and $\omega_Y$ lies in~$W_Y$,  we have $\omega_X(Y) \subseteq X$ and $\omega_X(\omega_Y)$ is a no trivial element in  $W_{\omega_X(Y)}$. Since $\omega_Y$ permutes $Y,$ the element $\omega_X(\omega_Y)$ must permut $\omega_X(Y)$.  So   $\omega_X(\omega_Y) = \omega_{\omega_Y(Z)}$. This proves Point~(i). Point~(ii) is proven by repeatedly applying Point~(i):  Assume $Z \subseteq Y\subseteq X$. Then $\omega_Y(Z) \subseteq Y$ and 
$\omega_{\omega_X(Y)}(\omega_X(\omega_Z)) = \omega_X(\omega_Y)(\omega_X(\omega_Z)) = (\omega_X\omega_Y\omega_X)(\omega_X\omega_Z\omega_X)(\omega_X\omega_Y\omega_X) =\omega_X\omega_Y\omega_Z\omega_Y\omega_X  = \omega_X(\omega_Y(\omega_Z)) = \omega_X(\omega_{\omega_Y(Z)}) =\omega_{\omega_X(\omega_Y(Z))}$. On the other hand, $\omega_{\omega_X(Y)}(\omega_X(\omega_Z)) = \omega_{\omega_X(Y)}(\omega_{\omega_X(Z)}) =  \omega_{\omega_{\omega_X(Y)}(\omega_X(Z))}$. So 
$\omega_{\omega_X(\omega_Y(Z))} = \omega_{\omega_{\omega_X(Y)}(\omega_X(Z))}$. But  for a given  Coxeter system $\WS$, any two minimal length representative words of the same element  of $W$ are written on the same letters (see \cite{Bou} for instance). This set of letters is called the support of the element. In particular,  for any~$X\in\FFF$ the support of $\omega_X$ is $X$. So,  $\omega_X(\omega_Y(Z)) = \omega_{\omega_X(Y)}(\omega_X(Z))$.

\end{proof}

\begin{lm}[Step 1] \label{step3} We still have a presentation of $C\WS$ by removing from the presentation~(\ref{pres1cactus})  all the relations of type ($R_2$) so that $X$ is not in $\CSLs$.
\end{lm}

\begin{proof}  Consider a relation $c_X c_Y =  c_{\omega_X(Y)} c_X$ of type ($R_2$)  that appears in the presentation~(\ref{pres1cactus})  where $X$ is not in $\CSLs$.  Since $X$ is not in $\CSLs$, it  has to belong to $\Omega_0(\overline{X})$ and the relation $c_{\overline{X}}c_{\rond{X}} = c_Xc_{\overline{X}}$ is of type $R_{2}$ with $\overline{X}$ in $\CSLs$. Moreover, both $Y$ and $\omega_X(Y)$  are included in $X$,  and so have to belong to $\Omega_0(\overline{X})$. Then he relations $c_{\overline{X}} c_Y= c_{\omega_{\overline{X}}(Y)}c_{\overline{X}} $ and $c_{\overline{X}}c_{\omega_X(Y)}=   c_{\omega_{\overline{X}}(\omega_X(Y))} c_{\overline{X}}$  are of type ($R_2$).   Using the relations of type $(R_{1})$, we see that the relation $c_X c_Y =  c_{\omega_X(Y)} c_X$ is equivalent to the relation $c_{\rond{X}}c_{\overline{X}} c_Yc_{\overline{X}} =  c_{\overline{X}}c_{\omega_X(Y)}c_{\overline{X}} c_{\rond{X}}$, which in turn is equivalent to  the relation $c_{\rond{X}}c_{\omega_{\overline{X}}(Y)} =  c_{\omega_{\overline{X}}(\omega_X(Y))} c_{\rond{X}}$ thanks to the two above relations (and relations of type ($R_1$)).  This relation can be written as $c_{\rond{X}}c_{\omega_{\overline{X}}(Y)} =  c_{\omega_{\omega_{\rond{X}}}(\omega_{\overline{X}}(Y))} c_{\rond{X}}$  by Lemma~\ref{lemme1}. By assumption $Y$ belongs  to $\Omega_0(X)$. This imposes  that $\omega_{\overline{X}}(Y)$ belongs to $\Omega_0(\rond{X})$. So the latter relation is of type ($R_2$) with $X_0$ in~$\CSLs$.  \end{proof}

\begin{lm}[Step 2] \label{step1} Starting with the  the presentation obtained at Step 1, we still have a presentation of $C\WS$ by removing  all the relations $c_X^2=1$ of type ($R_{1}$) with $X\in \FFF\setminus \Lambda$ and  by replacing the set of relations of type~$(\widehat{R}_{2})$ with the set of relations of type~$(\widehat{R}_{2.c})$. 
\end{lm}

\begin{proof} All generators have order $2$ by relations of type ($R_1$)  and  all $\omega_X$ have also of order~$2$. Therefore any relation $c_X c_Y =  c_{\omega_X(Y)} c_X$  of  type~($R_2$) is equivalent to the relation $c_Xc_{\omega_X(Y)} =  c_Y c_X$, of  type~($R_2$),  to the relation~$c_X c_Yc_X =  c_{\omega_X(Y)}$  and to the relation~$c_Xc_{\omega_X(Y)}c_X =  c_Y$, using relations of type ($R_1$) only.  So any relation of type~($\widehat{R}_2$) is  equivalent to a relation of type~($\widehat{R}_{2.c}$) using relations of type ($R_1$). Thereby, the set of relations of type ($\widehat{R}_2$)  can be replace with the set of relations  of type~($\widehat{R}_{2.c}$) in  the presentation (\ref{pres1cactus}). Now, when $X$ is not in $\CSLs$, the relation $c_X^2=1$ follows from the relation $c_{\rond{X}}^2=1$ using the relation $c_{\overline{X}}^2=1$ and  the relation $c_{\overline{X}}\,c_{\rond{X}}c_{\overline{X}} = c_X $, that is  of type~($\widehat{R}_{2.c}$) . 
\end{proof}
 
\begin{lm}[Step 3]\label{step2} Among the relations of type ($R_3$) all the relations so that  neither $X$ nor $Y$ belongs to $\CSLs$ can be removed  from the presentation obtained at Step 2. We still have a presentation of $C\WS$ by replacing    the remaining relations of type ($R_3$) with the relations of type ($R_{3.a}$) and ($R_{3.b}$).  
\end{lm}

\begin{proof} Let $X,Y$ be in $\FFF$ with $Y\in \OWSp{X}$. Consider the corresponding relation  $c_Y c_X = c_X c_Y$ of type ($R_3$). By property~(i)(b) in Definition~\ref{defiCSL}, there is $X_0$ in $\CSLs$ so that $X,Y$ lie in $\OWS {X_0}$ and $\{\omega_{X_0}(X), \omega_{X_0}(Y)\}\cap \CSLs \neq\emptyset$.  The relations  $c_{X_0}c_Y = c_{\omega_{X_0}(Y)}c_{X_0}$ and $c_{X_0}c_X = c_{\omega_{X_0}(X)}c_{X_0}$  are relations of type ($R_3$) in Presentation~(\ref{pres1cactus}).  Up to exchange $X$ and $Y$,  we can without restriction assume that $\omega_{X_0}(X)$ lies in  $\CSLs$. On the other hand,~$Y\in \OWSp{X}$. This implies that we have also $\omega_{X_0}(Y)\in \OWSp{\omega_{X_0}(X)}$. Therefore, the relation $c_{\omega_{X_0}(Y)} c_{\omega_{X_0}(X)} = c_{\omega_{X_0}(X)} c_{\omega_{X_0}(Y)}$ is a relation  of type ($R_3$), too.  Now, it is immediate that the relation $c_Y c_X = c_X c_Y$ can be deduce from  the three above relations: $c_{X_0}c_{\omega_{X_0}(Y)} c_{\omega_{X_0}(X)}c_{X_0} = c_{X_0}c_{\omega_{X_0}(Y)} c_{X_0}c_{X} = c_{X_0}c_{X_0}c_{Y}c_{X} =c_{Y}c_{X}$ and, similarly  $c_{X_0}c_{\omega_{X_0}(X)} c_{\omega_{X_0)}(Y))}c_{X_0} = c_{X}c_{Y}$.  This proves the first part of the statement.  Now, consider the relation of  $c_Y c_X = c_X c_Y$ of type ($R_3$) with $X$ in $\Lambda$.  If $Y$ is also in $\Lambda$, then,  using the relations of type ($R_{1.a}$), the relation $c_Y c_X = c_X c_Y$  can be replace  with the relation $(c_Y c_X)^2 = 1$  which is of type ($R_{3.a}$).
 If $Y$ is not in $\Lambda$, then $c_Y c_X = c_X c_Y$  can be replace  with the relation   $c_{\overline{Y}}c_{\rond{Y}}c_{\overline{Y}} c_X = c_X c_{\overline{Y}}c_{\rond{Y}}c_{\overline{Y}}$, which can, in turn, be replace with $(c_{\overline{Y}}c_{\rond{Y}}c_{\overline{Y}} c_X)^2 =1$, which is a relation of type ($R_{3.b}$). On the other hand, all relations of type~($R_{3.a}$)  or of type~($R_{3.b}$) are obtained in this way. 
\end{proof}

% Among the relations so that  $X$ belongs to $\CSLs$ and only one among  and $Y$ and $\omega_X(Y)$ belongs to $\CSLs$,  by the previous remark, we can keep only those so that $Y$ belong to $\CSLs$. We say that this relations are of type~($R_{2.c}$).   If only $X$  belongs to $\CSLs$, then  the relations  $c_{\overline{Y}} c_{Y} =  c_{\rond{Y}} c_{\overline{Y}}$ and $c_{\overline{\omega_X(Y)}} c_{\omega_X(Y)} =  c_{\rond{\omega_X(Y)}} c_{\overline{\omega_X(Y)}}$ are of type ${R_{2.c}}$

\begin{thm}\label{ThmA} Let $\WS$ be a Coxeter system and $(\CSLs,\Psi)$ be a section. Then $C\WS$ possesses the following group presentation: 
$$\left\langle C_\CSLs \left | \begin{array}{cccl}
 (R_{1.a})&c_X^2=1&;&X\in \CSLs\\
 (R_{2.a})& c_Xc_Y =c_{\omega_X(Y)} c_X &;&  X,Y,\omega_X(Y)\in \CSLs  \textrm{ with } Y\in \Omega_0(X)\\
 (R_{2.b})& c_Xc_Yc_Zc_Y =c_{Y'}c_{Z'}c_{Y'} c_X &;&\left\{\begin{array}{l}X\in \CSLs;\   (Y,Z), (Y',Z') \in \Psi(\Omega_0(X)\setminus \CSLs)\\ 
 \textrm{and } \omega_{Y'}(Z') = \omega_{X}(\omega_{Y}(Z))\end{array}\right.\\
 (R_{2.c})& c_Xc_Yc_X = c_{\overline{Z}} c_{\rond{Z}}c_{\overline{Z}}&;&X,Y \in \CSLs \textrm{, } Y \in \OWSs{X} \textrm{, } Z = \omega_X(Y)  \textrm{ and } (X,Y)\!\neq\!\Psi(Z) \\
 (R_{3.a})& (c_Xc_Y)^2 = 1&;&  X,Y \in \CSLs \textrm{ with } Y \in \OWSp{X}\\
 (R_{3.b})&(c_X c_Yc_Xc_Z)^2 = 1&;&  Z\in \CSLs \textrm{ and }  (X,Y)\in \Psi\big(\OWSp{Z}\setminus \CSLs\big) \end{array}\right. \right\rangle
$$
\end{thm}

\begin{proof}
Applying the above lemmas, we get that $C\WS$ has the following presentation:   
$$
\left\langle c_\FFF  \left | \begin{array}{cccl}
 (R_{1.a})&c_X^2=1&;&X\in \CSLs;\\
 (R_{2.a})& c_Xc_Y =c_{\omega_X(Y)} c_X &;&  X,Y,\omega_X(Y)\in \CSLs  \textrm{ with } Y\in \Omega_0(X);\\
(\widehat{R}_{2.b})& c_Xc_Y =c_{\omega_X(Y)} c_X  &;& X\in \CSLs \textrm{ and } Y, \omega_X(Y) \in \Omega_0(X) \setminus \CSLs; \\
(\widehat{R}_{2.c})& c_X c_Yc_X =  c_{\omega_X(Y)}&;& X,Y\in \CSLs  \textrm{ with } Y\in \OWSs{X} \textrm{ and } \omega_X(Y)\not\in\CSLs.\\ 
(R_{3.a})& (c_Xc_Y)^2 = 1&;&  X,Y \in \CSLs \textrm{ with } Y \in \OWSp{X};\\
 (R_{3.b})&(c_Y c_Zc_Yc_X)^2 = 1&;&  X\in \CSLs \textrm{ and }  (Y,Z)\in \Psi\big(\OWSp{X}\setminus \CSLs\big); \end{array}\right. \right\rangle$$

Let $X$ be in $\FFF\setminus \CSLs$.  Then the relation $c_{\overline{X}}c_{\rond{X}}c_{\overline{X}} = c_X$ is a relation of type $(\widehat{R}_{2,c})$. So, all such generators $c_X$ can be removed from the presentation,  all such  relations  $c_{\overline{X}}c_{\rond{X}}c_{\overline{X}} = c_X$ can also be removed from the presentation and all such letters $c_X$ can be replace with the word $c_{\overline{X}}c_{\rond{X}}c_{\overline{X}}$ in any remaining relation where these letters occur.
\end{proof}

Now Theorem~\ref{ThmA2} follows easily from Theorem~\ref{ThmA}: If  $\CSLs$ is a cross section,  then there is no relation of type  $(R_{2a})$,  $(R_{2,c})$ or  $(R_{3,a})$ (note that, in $\FFF$, in particular, for $X,Y,Z\in \CSLs$, if $\omega_X(Y) = Z$ then $X = Y = Z$).   

\begin{rmq} Consider a section~$(\CSLs,\Psi)$. When $\CSLs$ is a cross section then it is a transversal for~$\equiv$. If~$\CSLs$ is not a cross section but $(\CSLs,\Psi)$ is a transversal section  then in relations of type~$(R_{2,a})$ of Theorem~\ref{ThmA} we must have $\omega_X(Y) = Y$, and we get  commutation relations ; In relations of type~($R_{2,c}$) we must have $Y = \rond{Z}$.  Such a situation occurs in the case of a Coxeter group of type~$D_{2n+1}$  (see the next section). 
\end{rmq}

\section{Cross sections for finite irreducible Coxeter groups}
 \label{annexepresent}
 
 In this section we prove Proposition~\ref{propA}.  We recall that a Coxeter system $\WS$, and its Coxeter matrix, can be defined  by the corresponding Coxeter graph, which is a finite simple labelled graph $\Gamma$ whom vertex set is $S$ and so that any two vertices $s,t$ are joined by an edge when $m_{s,t} \geq 3$. in this case the edge is labelled with $m_{s,t}$.  The common convention when representing the graph is to omit the label when its value is 3.  Finite irreducible Coxeter groups are classified  by their graphs whom list is recalled below. 
 
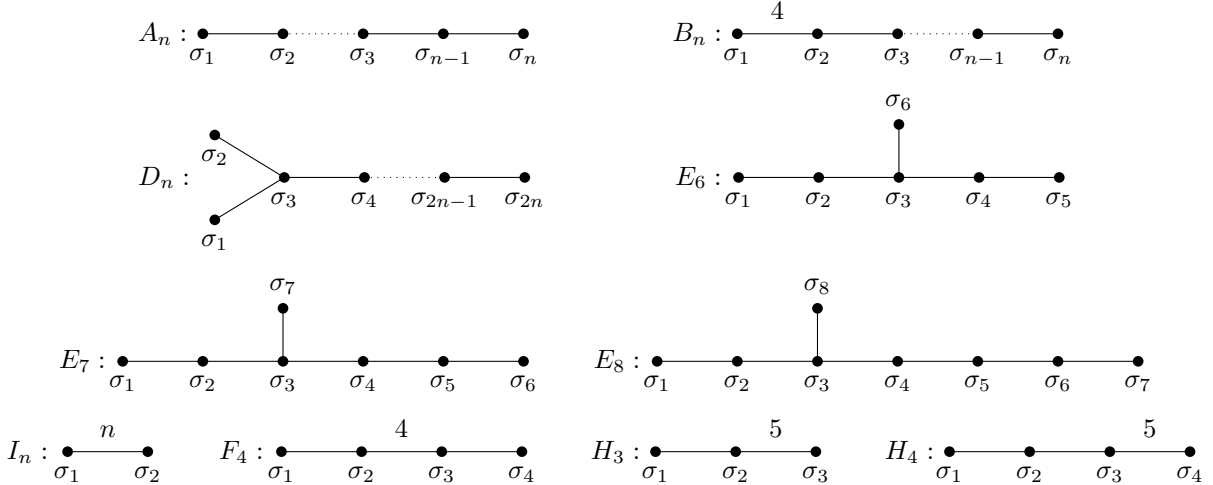
\begin{center}
\begin{tikzpicture}[x=20pt,y=20pt]
\fill(0,0) circle (0.1) node[below=2pt]{$\sigma_1$};
\fill(1.5,0) circle (0.1) node[below=2pt]{$\sigma_2$};
\fill(3,0) circle (0.1) node[below=2pt]{$\sigma_3$};
\fill(4.5,0) circle (0.1) node[below=2pt]{$\sigma_{n-1}$};
\fill(6,0) circle (0.1) node[below=2pt]{$\sigma_{n}$};
\draw (0,0) -- (1.5,0);
\draw[dotted] (1.5,0) -- (3,0);
\draw (3,0) -- (6,0);
\draw (-0.75,0) node {$A_n: $};
\fill(10,0) circle (0.1) node[below=2pt]{$\sigma_1$};
\fill(11.5,0) circle (0.1) node[below=2pt]{$\sigma_2$};
\fill(13,0) circle (0.1) node[below=2pt]{$\sigma_3$};
\fill(14.5,0) circle (0.1) node[below=2pt]{$\sigma_{n-1}$};
\fill(16,0) circle (0.1) node[below=2pt]{$\sigma_{n}$};
\draw (10,0) -- (13,0);
\draw[dotted] (13,0) -- (14.5,0);
\draw (14.5,0) -- (16,0);
\draw (10.75,0) node[above=2pt] {$4$};
\draw (9.25,0) node {$B_n: $};
\end{tikzpicture}
\end{center}

\begin{center}
\begin{tikzpicture}[x=20pt,y=20pt]
\fill(0.2,0.8) circle (0.1) node[below=2pt]{$\sigma_2$};
\fill(0.2,-0.8) circle (0.1) node[below=2pt]{$\sigma_1$};
\fill(1.5,0) circle (0.1) node[below=2pt]{$\sigma_3$};
\fill(3,0) circle (0.1) node[below=2pt]{$\sigma_4$};
\fill(4.5,0) circle (0.1) node[below=2pt]{$\sigma_{2n-1}$};
\fill(6,0) circle (0.1) node[below=2pt]{$\sigma_{2n}$};
\draw (0.2,0.8) -- (1.5,0) -- (3,0);
\draw (0.2,-0.8) -- (1.5,0);
\draw[dotted] (3,0) -- (4.5,0);
\draw (4.5,0) -- (6,0);
\draw (-0.75,0) node {$D_n: $};

\fill(10,0) circle (0.1) node[below=2pt]{$\sigma_1$};
\fill(11.5,0) circle (0.1) node[below=2pt]{$\sigma_2$};
\fill(13,0) circle (0.1) node[below=2pt]{$\sigma_3$};
\fill(14.5,0) circle (0.1) node[below=2pt]{$\sigma_4$};
\fill(16,0) circle (0.1) node[below=2pt]{$\sigma_5$};
\fill(13,1) circle (0.1) node[above=2pt]{$\sigma_6$};
\draw (10,0) -- (14.5,0);
\draw(13,1) -- (13,0);
\draw (14.5,0) -- (16,0);
\draw (9.25,0) node {$E_6: $};
\end{tikzpicture}
\end{center}

\begin{center}
\begin{tikzpicture}[x=20pt,y=20pt]
\fill(0,0) circle (0.1) node[below=2pt]{$\sigma_1$};
\fill(1.5,0) circle (0.1) node[below=2pt]{$\sigma_2$};
\fill(3,0) circle (0.1) node[below=2pt]{$\sigma_3$};
\fill(4.5,0) circle (0.1) node[below=2pt]{$\sigma_4$};
\fill(6,0) circle (0.1) node[below=2pt]{$\sigma_5$};
\fill(3,1) circle (0.1) node[above=2pt]{$\sigma_7$};
\fill(7.5,0) circle (0.1) node[below=2pt]{$\sigma_6$};
\draw (0,0) -- (7.5,0);
\draw(3,1) -- (3,0);
\draw (-0.75,0) node {$E_7: $};

\fill(10,0) circle (0.1) node[below=2pt]{$\sigma_1$};
\fill(11.5,0) circle (0.1) node[below=2pt]{$\sigma_2$};
\fill(13,0) circle (0.1) node[below=2pt]{$\sigma_3$};
\fill(14.5,0) circle (0.1) node[below=2pt]{$\sigma_4$};
\fill(16,0) circle (0.1) node[below=2pt]{$\sigma_5$};
\fill(13,1) circle (0.1) node[above=2pt]{$\sigma_8$};
\fill(17.5,0) circle (0.1) node[below=2pt]{$\sigma_6$};
\fill(19,0) circle (0.1) node[below=2pt]{$\sigma_7$};
\draw (10,0) -- (19,0);
\draw(13,1) -- (13,0);
\draw (9.25,0) node {$E_8: $};
\end{tikzpicture}
\end{center}

\begin{center}
\begin{tikzpicture}[x=20pt,y=20pt]
\fill(0,0) circle (0.1) node[below=2pt]{$\sigma_1$};
\fill(1.5,0) circle (0.1) node[below=2pt]{$\sigma_2$};
\draw (0,0) -- (1.5,0);
\draw (0.75,0) node[above=2pt] {$n$};
\draw (-0.75,0) node {$I_n: $};

\fill(4,0) circle (0.1) node[below=2pt]{$\sigma_1$};
\fill(5.5,0) circle (0.1) node[below=2pt]{$\sigma_2$};
\fill(7,0) circle (0.1) node[below=2pt]{$\sigma_3$};
\fill(8.5,0) circle (0.1) node[below=2pt]{$\sigma_{4}$};
\draw (4,0) -- (8.5,0);
\draw (6.25,0) node[above=2pt] {$4$};
\draw (3.25,0) node {$F_4: $};

\fill(11,0) circle (0.1) node[below=2pt]{$\sigma_1$};
\fill(12.5,0) circle (0.1) node[below=2pt]{$\sigma_2$};
\fill(14,0) circle (0.1) node[below=2pt]{$\sigma_3$};
\draw (11,0) -- (14,0);
\draw (13.25,0) node[above=2pt] {$5$};
\draw (10.25,0) node {$H_3: $};

\fill(16.5,0) circle (0.1) node[below=2pt]{$\sigma_1$};
\fill(18,0) circle (0.1) node[below=2pt]{$\sigma_2$};
\fill(19.5,0) circle (0.1) node[below=2pt]{$\sigma_3$};
\fill(21,0) circle (0.1) node[below=2pt]{$\sigma_{4}$};
\draw (16.5,0) -- (21,0);
\draw (20.25,0) node[above=2pt] {$5$};
\draw (15.75,0) node {$H_4: $};
\end{tikzpicture}
\end{center}

For each irreducible finite Coxeter  group $W$ with generating set $S$ that is not of type $E$, we provide either a cross section or a transversal section. The verifications are straightforward and are left to the reader. In the sequel when considering a finite irreducible Coxeter group of a given type, we use the above notations.  The reader may note that when $\omega_X(Y) = Z$ with $Z\neq Y$, then either $Y$ is of type $A$ or $Y$ type $D_5$ with $X$ is of type $E_6$. So, if $Y\in \FFF$ is not of type $A$ or of type $D_5$, it is alone in its $\equiv$-class and has to belong to the section~$\CSLs$. For the same reason, if $\omega_S$ is central, then each maximal element of $\FFF\setminus\{S\}$ has to belong to the section.

\subsection{Type $A_n$ }

Consider the Coxeter system $(W,S)$ of type $A_n$. Then, the following set is a cross section: $$\Lambda =\bigg \{\{\sigma_1,\cdots, \sigma_j\}\mid 1\leq j\leq n \bigg\}
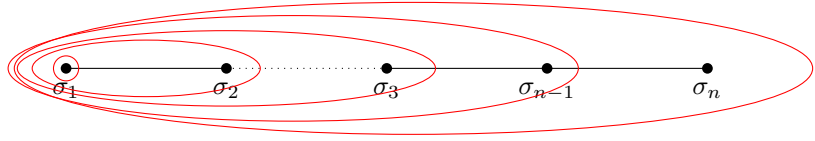
$$ 
\begin{center}
\begin{tikzpicture}[x=20pt,y=20pt]
\fill(0,0) circle (0.1) node[below=2pt]{$\sigma_1$};
\fill(3,0) circle (0.1) node[below=2pt]{$\sigma_2$};
\fill(6,0) circle (0.1) node[below=2pt]{$\sigma_3$};
\fill(9,0) circle (0.1) node[below=2pt]{$\sigma_{n-1}$};
\fill(12,0) circle (0.1) node[below=2pt]{$\sigma_{n}$};
\draw (0,0) -- (3,0);
\draw[dotted] (3,0) -- (6,0);
\draw (6,0) -- (12,0);
\node[ellipse,
    draw = red,
    minimum width = .2cm, 
    minimum height = 0.25cm] (e) at (0,0) {};
    \node[ellipse,
    draw = red,
    minimum width = 3cm, 
    minimum height = .75cm] (e) at (1.5,0) {};
    \node[ellipse,
    draw = red,
    minimum width = 5.5cm, 
    minimum height = 1cm] (e) at (3,0) {};
    \node[ellipse,
    draw = red,
    minimum width = 7.5cm, 
    minimum height = 1.4cm] (e) at (4.25,0) {};
    \node[ellipse,
    draw = red,
    minimum width = 10.5cm, 
    minimum height = 1.75cm] (e) at (6.5,0) {};
\end{tikzpicture}
\end{center}

 The associated presentation given in  Theorem~\ref{ThmA2} is the one provides in \cite{CheNan2025}. 
\subsection{Type $B_n$ }
Consider the Coxeter system $(W,S)$ of type $B_n$.  Then, Then, the following set is a cross section:  $$\Lambda = \bigg\{\{\sigma_1,\cdots, \sigma_j\}\mid 1\leq j\leq n\bigg\}\cup \bigg\{\{\sigma_j,\cdots, \sigma_n\}\mid 2\leq j\leq n\bigg\}
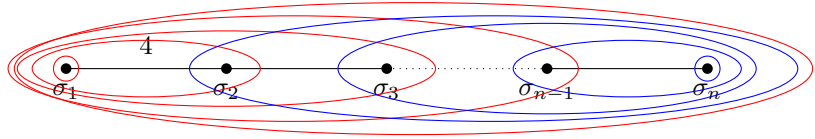
$$
\begin{center}
\begin{tikzpicture}[x=20pt,y=20pt]
\fill(0,0) circle (0.1) node[below=2pt]{$\sigma_1$};
\fill(3,0) circle (0.1) node[below=2pt]{$\sigma_2$};
\fill(6,0) circle (0.1) node[below=2pt]{$\sigma_3$};
\fill(9,0) circle (0.1) node[below=2pt]{$\sigma_{n-1}$};
\fill(12,0) circle (0.1) node[below=2pt]{$\sigma_{n}$};
\draw (0,0) -- (6,0);
\draw[dotted] (6,0) -- (9,0);
\draw (9,0) -- (12,0);
\draw (1.5,0) node[above=2pt] {$4$};

  \node[ellipse,
    draw = red,
    minimum width = .2cm, 
    minimum height = 0.25cm] (e) at (0,0) {};
    \node[ellipse,
    draw = red,
    minimum width = 3cm, 
    minimum height = .75cm] (e) at (1.5,0) {};
    \node[ellipse,
    draw = red,
    minimum width = 5.5cm, 
    minimum height = 1cm] (e) at (3,0) {};
    \node[ellipse,
    draw = red,
    minimum width = 7.5cm, 
    minimum height = 1.4cm] (e) at (4.25,0) {};
    \node[ellipse,
    draw = red,
    minimum width = 10.5cm, 
    minimum height = 1.75cm] (e) at (6.5,0) {}; 
    
    \node[ellipse,
    draw = blue,
    minimum width = .2cm, 
    minimum height = 0.25cm] (e) at (12,0) {};
    \node[ellipse,
    draw = blue,
    minimum width = 3cm, 
    minimum height = 0.75cm] (e) at (10.5,0) {};
    \node[ellipse,
    draw = blue,
    minimum width = 8cm, 
    minimum height = 1.4cm] (e) at (8,0) {};
    \node[ellipse,
    draw = blue,
    minimum width = 5.5cm, 
    minimum height = 1.2cm] (e) at (9,0) {};
   
\end{tikzpicture}
\end{center}

 \subsection{Type $D_n$}

The result depends on whether $n$ is even or odd. 

Consider the Coxeter system $(W,S)$ of type $D_{2n}$ with $n\geq 2$. In this case $\omega_S$ fixes $S$ and the following set is a cross section:  $$\Lambda = \bigg\{\{\sigma_1,\cdots, \sigma_j\}\mid 1\leq j\leq 2n\bigg\}\cup \bigg\{\{\sigma_j,\cdots, \sigma_{2n}\}\mid 2\leq j\leq 2n\bigg\} \cup \bigg\{\{\sigma_1\}\cup \{\sigma_3,\cdots, \sigma_{2n}\} \bigg\}
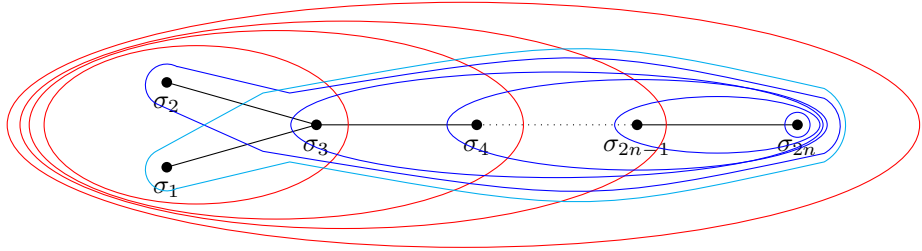
$$
\begin{center}
\begin{tikzpicture}[x=20pt,y=20pt]
\fill(0.2,0.8) circle (0.1) node[below=2pt]{$\sigma_2$};
\fill(0.2,-0.8) circle (0.1) node[below=2pt]{$\sigma_1$};
\fill(3,0) circle (0.1) node[below=2pt]{$\sigma_3$};
\fill(6,0) circle (0.1) node[below=2pt]{$\sigma_4$};
\fill(9,0) circle (0.1) node[below=2pt]{$\sigma_{2n-1}$};
\fill(12,0) circle (0.1) node[below=2pt]{$\sigma_{2n}$};
\draw (0.2,0.8) -- (3,0) -- (6,0);
\draw (0.2,-0.8) -- (3,0);
\draw[dotted] (6,0) -- (9,0);
\draw (9,0) -- (12,0);

    \node[ellipse,
    draw = red,
    minimum width = 4cm, 
    minimum height = 2.1cm] (e) at (0.75,0) {};
    \node[ellipse,
    draw = red,
    minimum width = 6.5cm, 
    minimum height = 2.5cm] (e) at (2.25,0) {};
    \node[ellipse,
    draw = red,
    minimum width = 8.5cm, 
    minimum height = 2.75cm] (e) at (3.5,0) {};
    \node[ellipse,
    draw = red,
    minimum width = 12.cm, 
    minimum height = 3.25cm] (e) at (5.75,0) {}; 

\draw[draw = blue]  (0.4,1.1)--(2.5,0.6)..controls (8,1.5)..(12.5,0.5) arc(60:-60:0.6)..controls (8,-1.5)..(2,-0.5)--(0,0.4) arc (-120: -300:0.4); 

\draw[draw = cyan]  (0.4,-1.2)--(2.5,-0.7)..controls (8,-1.7)..(12.5,-0.7) arc(-60:60:0.8)..controls (8,1.7)..(2,0.6)--(0,-0.5) arc (120:300:0.4); 
   
   \node[ellipse,
    draw = blue,
    minimum width = .2cm, 
    minimum height = 0.25cm] (e) at (12,0) {};
    \node[ellipse,
    draw = blue,
    minimum width = 2.7cm, 
    minimum height = 0.75cm] (e) at (10.5,0) {};
    \node[ellipse,
    draw = blue,
    minimum width = 7cm, 
    minimum height = 1.4cm] (e) at (7.5,0) {};
    \node[ellipse,
    draw = blue,
    minimum width = 5.cm, 
    minimum height = 1.2cm] (e) at (9,0) {};

\end{tikzpicture}
\end{center}

Consider the Coxeter system $(W,S)$ of type $D_{2n}$ with $n\geq 2$.  In this case $\omega_S$ exchanges $\sigma_1$ and $\sigma_2$, and fixes the other generators.  Then, the following set is a transversal section:  $$\Lambda = \bigg\{\{\sigma_1,\cdots, \sigma_j\}\mid 1\leq j\leq 2n+1\bigg\}\cup \bigg\{\{\sigma_2,\cdots, \sigma_{j}\}\mid 2\leq j\leq 2n+1\bigg\}$$  This does not provide a cross section, because for every subset $X = \{\sigma_1,\cdots, \sigma_{2k+1} \}$ with $k\geq 1$,  the element $\omega_X$ exchanges $\sigma_1$ and $\sigma_2$.  One can verify that no cross section exist. 

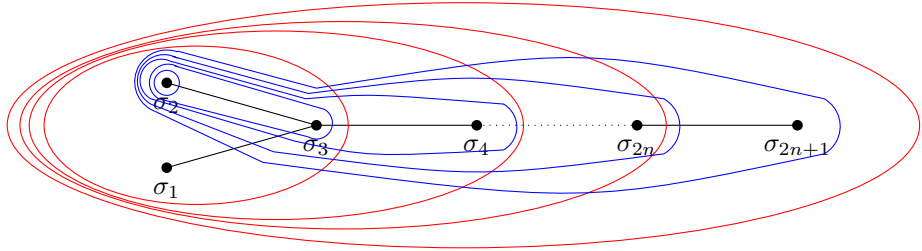
\begin{center}
\begin{tikzpicture}[x=20pt,y=20pt]
\fill(0.2,0.8) circle (0.1) node[below=2pt]{$\sigma_2$};
\fill(0.2,-0.8) circle (0.1) node[below=2pt]{$\sigma_1$};
\fill(3,0) circle (0.1) node[below=2pt]{$\sigma_3$};
\fill(6,0) circle (0.1) node[below=2pt]{$\sigma_4$};
\fill(9,0) circle (0.1) node[below=2pt]{$\sigma_{2n}$};
\fill(12,0) circle (0.1) node[below=2pt]{$\sigma_{2n+1}$};
\draw (0.2,0.8) -- (3,0) -- (6,0);
\draw (0.2,-0.8) -- (3,0);
\draw[dotted] (6,0) -- (9,0);
\draw (9,0) -- (12,0);

   \node[ellipse,
    draw = red,
    minimum width = 4cm, 
    minimum height = 2.1cm] (e) at (0.75,0) {};
    \node[ellipse,
    draw = red,
    minimum width = 6.5cm, 
    minimum height = 2.5cm] (e) at (2.25,0) {};
    \node[ellipse,
    draw = red,
    minimum width = 8.5cm, 
    minimum height = 2.75cm] (e) at (3.5,0) {};
    \node[ellipse,
    draw = red,
    minimum width = 12.cm, 
    minimum height = 3.25cm] (e) at (5.75,0) {}; 

\draw[draw = blue]  (0.4,1.4)--(3,0.7)..controls (8,1.5)..(12.5,0.5) arc(60:-60:0.6)..controls (8,-1.5)..(2,-0.7)--(-0.1,0.3) arc (-120: -290:0.6); 
\draw[draw = blue]  (0.3,1.3)--(2.85,0.6)..controls (6,1)..(9.5,0.5) arc(60:-60:0.6)..controls (6,-1)..(2.2,-0.5)--(-0.1,0.4) arc (-120: -300:0.5); 
\draw[draw = blue]  (0.3,1.2)--(2.8,0.5)..controls (4,0.6)..(6.5,0.4) arc(60:-60:0.5)..controls (4,-0.6)..(2.6,-0.3)--(-0.1,0.5) arc (-120: -300:0.4); 
\draw[draw = blue]  (0.35,1.12)--(3.15,0.3) arc(60:-120:0.3)--(0.05,0.5) arc (-120: -300:0.35); 
   \node[ellipse,
    draw = blue,
    minimum width = .2cm, 
    minimum height = 0.2cm] (e) at (0.2,0.8) {};

\end{tikzpicture}
\end{center}

 \subsection{Type $E_6$, $E_7$, $E_8$}
 
  Consider a Coxeter system $(W,S)$ of type $E_6$, $E_7$ or $E_8$ .  Then,  no transversal section exists. Indeed, consider the type $E_6$.  There is 5 elements in $\FFF$ that are of type $A_3$ and they are all in the same $\equiv$-class.  Consider a section $\CSLs$.  To be transversal, the section $\CSLs$ has to contain  only one  of these five elements and this element must be able to be send to the 4 others using some $\omega_X$  with $X$ in~$\CSLs$. The unique possibility is  that $\{\sigma_2,\sigma_3,\sigma_4\}$ and that the two subgraphs of type $D_5$ belong to $\CSLs$. But the latter two subgraphs are send one to the other by $\omega_S$.  So $\CSLs$ can not be transversal. Clearly the same argument  can be applied for types $E_7$ and $E_8$. 
  \subsection{Type $I_n$} Consider a Coxeter system $(W,S)$ of type $I_n$ with $n\geq 3$. The result depends on whether $n$ is even or odd.  If $n$ is even, then $\FFF$ is itself a cross section. If $n$ is odd, then $\bigg\{ \{\sigma_1\}; S\bigg\}$ is a cross section.     
  
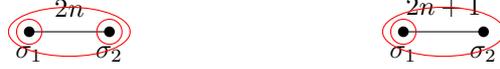
\begin{center}
\begin{tikzpicture}[x=20pt,y=20pt]
\fill(0,0) circle (0.1) node[below=2pt]{$\sigma_1$};
\fill(1.5,0) circle (0.1) node[below=2pt]{$\sigma_2$};
\draw (0,0) -- (1.5,0);
\draw (0.75,0) node[above=2pt] {$2n$};

\fill(7,0) circle (0.1) node[below=2pt]{$\sigma_1$};
\fill(8.5,0) circle (0.1) node[below=2pt]{$\sigma_2$};
\draw (7,0) -- (8.5,0);
\draw (7.75,0) node[above=2pt] {$2n+1$};

\node[ellipse,
    draw = red,
    minimum width = .2cm, 
    minimum height = 0.2cm] (e) at (0,0) {};
    
\node[ellipse,
    draw = red,
    minimum width = .2cm, 
    minimum height = 0.2cm] (e) at (7,0) {};    

\node[ellipse,
    draw = red,
    minimum width = .2cm, 
    minimum height = 0.2cm] (e) at (1.5,0) {};    
    
\node[ellipse,
    draw = red,
    minimum width = 1.6cm, 
    minimum height = 0.65cm] (e) at (0.75,0) {};    
\node[ellipse,
    draw = red,
    minimum width = 1.6cm, 
    minimum height = 0.65cm] (e) at (7.75,0) {};   
\end{tikzpicture}

\end{center}

  \subsection{Type $F_4$}  
  
  Consider the Coxeter system $(W,S)$ of type $F_4$. Then, the set $\FFF\setminus \bigg\{ \{\sigma_1\}; \{\sigma_4\}  \bigg\}$ is a cross section.
  
  
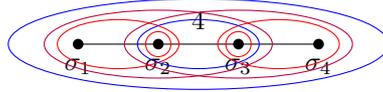
\begin{center}
\begin{tikzpicture}[x=20pt,y=20pt]
\fill(0,0) circle (0.1) node[below=2pt]{$\sigma_1$};
\fill(1.5,0) circle (0.1) node[below=2pt]{$\sigma_2$};
\fill(3,0) circle (0.1) node[below=2pt]{$\sigma_3$};
\fill(4.5,0) circle (0.1) node[below=2pt]{$\sigma_{4}$};
\draw (0,0) -- (4.5,0);
\draw (2.25,0) node[above=2pt] {$4$};
\node[ellipse,
    draw = red,
    minimum width = .2cm, 
    minimum height = 0.2cm] (e) at (1.5,0) {};    
    
\node[ellipse,
    draw = red,
    minimum width = 1.6cm, 
    minimum height = 0.65cm] (e) at (0.75,0) {};    

\node[ellipse,
    draw = red,
    minimum width = .2cm, 
    minimum height = 0.2cm] (e) at (3,0) {};    
    
\node[ellipse,
    draw = red,
    minimum width = 1.6cm, 
    minimum height = 0.65cm] (e) at (3.75,0) {};    

\node[ellipse,
    draw = blue,
    minimum width = 1.6cm, 
    minimum height = 0.65cm] (e) at (2.25,0) {};

\node[ellipse,
    draw = blue,
    minimum width = 5cm, 
    minimum height = 1.2cm] (e) at (2.25,0) {};  
    
\node[ellipse,
	draw = purple,
    minimum width = 3cm, 
    minimum height = 0.9cm] (e) at (1.5,0) {};  
\node[ellipse,
	draw = purple,
    minimum width = 3cm, 
    minimum height = 0.9cm] (e) at (3,0) {};  
\end{tikzpicture}
\end{center}

  \subsection{Type $H_3$ et $H_4$}  
  
Consider the Coxeter system $(W,S)$ of type $H_3$.  Then, the set $\FFF\setminus \bigg\{ \{\sigma_1\}; \{\sigma_3\}  \bigg\}$ is a cross section.

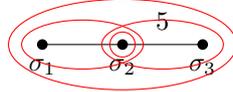
\begin{center}
\begin{tikzpicture}[x=20pt,y=20pt]
\fill(0,0) circle (0.1) node[below=2pt]{$\sigma_1$};
\fill(1.5,0) circle (0.1) node[below=2pt]{$\sigma_2$};
\fill(3,0) circle (0.1) node[below=2pt]{$\sigma_3$};
\draw (0,0) -- (3,0);
\draw (2.25,0) node[above=2pt] {$5$};
\node[ellipse,
    draw = red,
    minimum width = .2cm, 
    minimum height = 0.2cm] (e) at (1.5,0) {};    
    \node[ellipse,
    draw = red,
    minimum width = 1.6cm, 
    minimum height = 0.65cm] (e) at (0.75,0) {};   
    \node[ellipse,
    draw = red,
    minimum width = 1.6cm, 
    minimum height = 0.65cm] (e) at (2.25,0) {};  
    \node[ellipse,
    draw = red,
    minimum width = 3cm, 
    minimum height = 1.2cm] (e) at (1.5,0) {};  
\end{tikzpicture}
\end{center}

Consider the Coxeter system $(W,S)$ of type $H_4$. Then, the set  $$\Lambda = \bigg\{\{\sigma_j,\cdots, \sigma_3\}\mid 1\leq j\leq 3\bigg\}\cup \bigg\{\{\sigma_3,\sigma_4\} ; S \bigg\}$$ is a cross section for $\FFF$. 

  \begin{center}
\begin{tikzpicture}[x=20pt,y=20pt]
\fill(0,0) circle (0.1) node[below=2pt]{$\sigma_1$};
\fill(1.5,0) circle (0.1) node[below=2pt]{$\sigma_2$};
\fill(3,0) circle (0.1) node[below=2pt]{$\sigma_3$};
\fill(4.5,0) circle (0.1) node[below=2pt]{$\sigma_{4}$};
\draw (0,0) -- (4.5,0);
\draw (3.75,0) node[above=2pt] {$5$};
 \node[ellipse, 
 draw = red,
    minimum width = .2cm, 
    minimum height = 0.2cm] (e) at (3,0) {};    
    \node[ellipse,
    draw = red,
    minimum width = 1.6cm, 
    minimum height = 0.65cm] (e) at (2.25,0) {};   
    \node[ellipse,
    draw = red,
    minimum width = 1.6cm, 
    minimum height = 0.65cm] (e) at (3.55,0) {};  
    \node[ellipse,
    draw = red,
    minimum width = 4.5cm, 
    minimum height = 1.2cm] (e) at (2.25,0) {};  
    \node[ellipse,
	draw = red,
    minimum width = 3cm, 
    minimum height = 0.8cm] (e) at (1.5,0) {};  
\end{tikzpicture}
\end{center}

%\bibliographystyle{acm}
%\bibliography{biblio}
\end{document}